\documentclass[reqno]{amsart}
\usepackage{amsmath,amsthm,amscd,amsfonts,amssymb,enumerate}
\usepackage{graphicx}		
\usepackage{color}
\usepackage[colorlinks]{hyperref}
\usepackage{verbatim}
\usepackage{mathrsfs}
\usepackage{bm}

\newtheorem{theorem}{Theorem}[section]

\newtheorem{lemma}[theorem]{Lemma}

\theoremstyle{definition}

\theoremstyle{remark}

\numberwithin{equation}{section}

\newenvironment{proof3.2}{\medskip\noindent{\bf Proof of Theorem 3.2:}\enspace}{\hfill \qed \newline \medskip}
\newenvironment{proof3.3}{\medskip\noindent{\bf Proof of Theorem 3.3:}\enspace}{\hfill \qed \newline \medskip}
\newenvironment{proof4.2}{\medskip\noindent{\bf Proof of Theorem 4.2:}\enspace}{\hfill \qed \newline \medskip}
\begin{document}
\title[Asymptotic stability for a class of viscoelastic quations]
{Asymptotic stability for a class of viscoelastic equations
with general relaxation functions and the time delay}
\author[M. Liao, Z. Tan ]
{Menglan Liao, Zhong Tan$^*$}

\address{Menglan Liao \newline  
School of Mathematical Sciences, Xiamen University, Xiamen, Fujian, 361005, People's Republic of China }
\email{liaoml14@mails.jlu.edu.cn}
\address{Zhong Tan \newline  
School of Mathematical Sciences, Xiamen University, Xiamen, Fujian, 361005, People's Republic of China }
\email{tan85@xmu.edu.cn}

\thanks{Supported by the National Natural Science Foundation of China (Grant Nos. 11926316, 11531010, 12071391)}
\thanks{$^*$ Corresponding author: Zhong Tan}

\begin{abstract}
The goal of the present paper is to study the viscoelastic wave equation with the time delay
\[ |u_t|^\rho u_{tt}-\Delta u-\Delta u_{tt}+\int_0^tg(t-s)\Delta u(s)ds+\mu_1u_t(x,t)+\mu_2 u_t(x,t-\tau)=b|u|^{p-2}u\]
under initial boundary value conditions, where $\rho,~b,~\mu_1$ are positive constants, $\mu_2$ is a real number, $\tau>0$ represents the time delay. By using the multiplier method together with some properties of the convex functions, the explicit and general stability results of energy are proved under the  general assumption on the relaxation function $g$. This work generalizes and improves earlier results on the stability of the viscoelastic equations with the time delay in the literature.
\end{abstract}

\maketitle

\section{Introduction}
This paper concerns the following initial boundary value problem with the time delay
\begin{equation}
\label{1111}
\begin{cases}
      |u_t|^\rho u_{tt}-\Delta u-\Delta u_{tt}+\int_0^tg(t-s)\Delta u(s)ds\\
      \quad+\mu_1u_t(x,t)+\mu_2 u_t(x,t-\tau)=b|u|^{p-2}u& (x,t)\in\Omega \times(0,\infty),  \\
      u_t(x,t-\tau)=f_0(x,t-\tau)&(x,t)\in\Omega\times(0,\tau),  \\
      u(x,0)=u_0(x),~u_t(x,0)=u_1(x)&x\in\Omega,\\
      u(x,t)=0&(x,t)\in\partial\Omega\times[0,\infty),
\end{cases}
\end{equation}
where $\Omega\subset \mathbb{R}^N(N\ge 1)$ is a bounded domain with a smooth boundary $\partial\Omega$, the unknown $u := u(x, t)$ is a real valued function defined on $\Omega \times(0,\infty)$, $\rho,~b,~\mu_1$ are positive constants, $\mu_2$ is a real number, $\tau>0$ represents the time delay, $g$ is  the kernel of the memory term, and the initial data $(u_0,u_1,f_0)$ are given functions belonging to suitable spaces. In addition, the following assumptions are given throughout this paper: 

$(\mathbf{H_1})$  The relaxation function $g:[0,\infty)\to (0,\infty)$ is a differentiable function satisfying 
\begin{equation}
\label{02}
1-\int_0^\infty g(s)ds=l>0,
\end{equation}
and there exists a $C^1$ function $G:(0,\infty)\to (0,\infty)$ which is strictly increasing and strictly convex $C^2$ function on $(0,r]$, $r\le g(0)$, with $G(0)=G'(0)=0$ such that 
\begin{equation}
\label{add02}
g'(t)\le -\zeta(t) G(g(t)) \quad\text{for }t\ge 0,
\end{equation}
here $\zeta(t)$ is a positive non-increasing differentiable function.

$(\mathbf{H_2})$ $\rho$ and $p$ satisfy 
\[0<\rho\le \frac{2}{N-2}\text{ for }N\ge 3\text{ and }\rho>0\text{ for }N=1,2;\]
\[2<p\le \frac{2(N-1)}{N-2}\text{ for }N\ge 3\text{ and }p>2\text{ for }N=1,2.\]

It is well known that time delay effects which often appear in many practical applications may induce some instabilities. Some results on the local existence and blow-up of solutions to a class of equations with delay have been obtained, the interested readers can refer to \cite{CB2020,KM2016,KM2018,K2018,W2019} and the reference therein. Nicaise and  Pignotti \cite{NP2006} considered the wave equation with
a delay term in the boundary condition as well as the wave equation with a delayed velocity term and mixed Dirichlet-Neumann boundary condition in a bounded and smooth domain, respectively. By introducing suitable
energies and by using some observability inequalities, they proved an exponential stability of the solution in both cases under suitable assumptions. Kirane and Said-Houari \cite{KS2011} studied  the following viscoelastic wave equation with a delay term in internal feedback
\begin{equation}
\label{51}
u_{tt}-\Delta u+\int_0^tg(t-s)\Delta u(s)ds+\mu_1u_t(x,t)+\mu_2 u_t(x,t-\tau)=0
\end{equation}
here  $\mu_1,~\mu_2$ are positive constants, under the initial boundary conditions of problem \eqref{1111}. They proved the existence of a unique weak solution for $\mu_2\le \mu_1$ by using the Faedo-Galerkin approximations together with some energy estimates. Provided that  $g:\mathbb{R}^+\to \mathbb{R}^+$ is a $C^1$ function satisfying $g(0)>0$ and \eqref{02}, and there exists a positive non-increasing differentiable function $\zeta(t)$ such that
\begin{equation}
\label{52}
g'(t)\le-\zeta(t)g(t)\text{ for }t\ge 0\text{ and }\int_0^{+\infty} \zeta(t)dt=+\infty,
\end{equation}
by establishing suitable Lyapunov functionals, they also obtained the corresponding  exponential stability for $\mu_2<\mu_1$  and for $\mu_2=\mu_1$, respectively. Subsequently, Dai and Yang \cite{DY2014} proved an existence result of problem \eqref{51} without restrictions of $\mu_1,~\mu_2>0$ and $\mu_2\le \mu_1$. Making use of the viscoelasticity term controls the delay term, they also proved an energy decay result for problem \eqref{51} in the case $\mu_1=0$ provided that $g:\mathbb{R}^+\to \mathbb{R}^+$ is a $C^1$ function satisfying $g(0)>0$ and \eqref{02}, and there exists a positive constant $\zeta$ such that
\begin{equation}
\label{53}
g(t) \le -\zeta g(t)\text{ for }t>0.
\end{equation}
Liu \cite{L2013} generalized the results obtained by Kirane and Said-Houari \cite{KS2011}. That is, by the similar method in \cite{KS2011}, they established a general energy decay result for problem \eqref{51} with $\tau(t)$ instead of $\tau$.  In the absence of the source term $b|u|^{p-2}u$ in problem \eqref{1111}, Wu \cite{W2013} proved an energy decay by the similar method  in \cite{KS2011}, and generalized the results to the time-varying delay in \cite{W2016}. There are many papers concerning with the stability of  viscoelastic equations with time delay, the interested readers may refer to \cite{BSG2021,F2017,MK2016} and the reference therein. However, the  relaxation function $g$ are mainly limited to satisfying among the three conditions, which are \eqref{52}, \eqref{53} and that  $g:\mathbb{R}^+\to \mathbb{R}^+$ is a differentiable function satisfying $g(0)>0$ and \eqref{02}, and there exists a positive function $G\in C^1(\mathbb{R}^+)$ and $G$ is linear or strictly increasing and strictly convex $C^2$ function on $(0, r]$, $r< 1$, with $G(0) = G'(0) = 0$, such that
\[g'(t)\le -G(g(t))\text{ for }t>0.\]
Until recently,  Chellaoua and Boukhatem \cite{0CB2021} generalized the previous conditions that the relaxation function $g$ satisfied,  specifically investigated the following  second-order abstract viscoelastic equation in Hilbert spaces
\begin{equation*}
\begin{cases}
     u_{tt}+Au-\int_0^\infty g(s)Bu(t-s)ds+\mu_1u_t(t)+\mu_2 u_t(t-\tau)=0& t>0,  \\
      u_t(t-\tau)=f_0(t-\tau)&t\in(0,\tau),  \\
      u(-t)=u_0(t),~u_t(0)=u_1&t\ge 0,
\end{cases}
\end{equation*}
where  $A: D(A) \to H$ and $B : D(B) \to H$ are a self-adjoint linear positive operator with domains $D(A) \subset D(B)\subset H$ such that the embeddings are dense and compact. They established an explicit and general decay results of the energy solution by introducing a suitable Lyapunov functional and some properties of the convex functions under the condition $(\mathbf{H_1})$. Chellaoua and Boukhatem also addressed the stability results for the following second-order abstract viscoelastic equation in Hilbert spaces with
time-varying delay in \cite{CB2021}
\begin{equation*}
\begin{cases}
     u_{tt}+Au-\int_0^t g(t-s)Bu(s)ds+\mu_1u_t(t)+\mu_2 u_t(t-\tau(t))=0& t>0,  \\
      u_t(t-\tau(0))=f_0(t-\tau(0))&t\in(0,\tau(0)),  \\
      u(0)=u_0,~u_t(0)=u_1&t\ge 0,
\end{cases}
\end{equation*}
under the condition $(\mathbf{H_1})$. It is worth pointing out that Mustafa \cite{M2018} first proposed the condition $(\mathbf{H_1})$ to study the decay rates for the following initial boundary value problem 
\begin{equation*}
\begin{cases}
     u_{tt}-\Delta u+\int_0^t g(t-s)\Delta u(s)ds=0& \text{in }\Omega\times (0,\infty),  \\
      u=0&\text{on }\partial\Omega\times (0,\infty),  \\
      u(x,0)=u_0(x),~u_t(x,0)=u_1(x)&x\in\Omega.
\end{cases}
\end{equation*}
After that, many authors popularized the method used by Mustafa in \cite{M2018}. The readers may see the references \cite{BAM2020,FL2020,HM2019,0M2018,M2021} to get more details.

Motivated by the above works, we are committed to considering the stability of problem \eqref{1111} when the relaxation function $g$ satisfies the condition $(\mathbf{H_1})$. To the best of our knowledge,  there is no decay result for problem \eqref{1111} where the relaxation functions satisfy $(\mathbf{H_1})$, although Wu \cite{W2019} has investigated problem \eqref{1111} and proved the blow-up result with nonpositive and positive initial energy. With minimal conditions on the relaxation function $g$, we establish a general and optimal energy decay rates of problem \eqref{1111} in Theorem \ref{thm3.3}. Our proof is based on the multiplier method and the similar arguments in \cite{0CB2021,M2018} but it is different from before since the presence of $\Delta u_{tt}$, the time delay and the force source 
$b|u|^{p-2}u$. The outline of this paper is as follows: In Section 2, we give some preliminary results. Section 3 is used to present the energy decay and its proof. In Section 4, we give the possible generalizations.

\section{Preliminaries}
Throughout this paper, we denote by $\|\cdot\|_p$  and $\|\nabla \cdot\|_2$ the norm on $L^p(\Omega)$ with $1\le p\le \infty$ and $H_0^1(\Omega)$, respectively. Let $\lambda_1$ be the first eigenvalue of  the following boundary value problem 
\begin{equation*}
\begin{cases}
-\Delta \psi=\lambda\psi&\quad \text{for }x\in \Omega,\\
\psi=0&\quad\text{for }x\in \partial\Omega.
\end{cases}
\end{equation*} 
The symbol $c_s$ is the optimal embedding constant of $H_0^1(\Omega)\hookrightarrow L^p(\Omega)$.

In order to the completeness of results, in what follows, we state the known  results in \cite{W2019}.
Let us introduce the new variable 
\[z(x,\kappa,t)=u_t(x,t-\tau\kappa)\quad \text{ for }x\in\Omega,~\kappa\in (0,1),\]
then problem \eqref{1111} is equivalent to 
\begin{equation}
\label{2.1}
\begin{cases}
      |u_t|^\rho u_{tt}-\Delta u-\Delta u_{tt}+\int_0^tg(t-s)\Delta u(s)ds\\
      \quad+\mu_1u_t(x,t)+\mu_2 z(x,1,t)=b|u|^{p-2}u& (x,t)\in\Omega \times(0,\infty),  \\
      \tau z_t(x,\kappa,t)+z_\kappa(x,\kappa,t)=0&(x,t)\in\Omega\times(0,\infty),~\kappa\in (0,1),\\
      z(x,0,t)=u_t(x,t)&(x,t)\in\Omega\times(0,\infty),\\
      z(x,\kappa,0)=f_0(x,-\tau\kappa)&x\in\Omega,\\
      u_t(x,t-\tau)=f_0(x,t-\tau)&(x,t)\in\Omega\times(0,\tau),  \\
      u(x,0)=u_0(x),~u_t(x,0)=u_1(x)&x\in\Omega,\\
      u(x,t)=0&(x,t)\in\partial\Omega\times[0,\infty).
\end{cases}
\end{equation}

\begin{theorem}[Theorem 2.3 in \cite{W2019}]
Suppose that $|\mu_2|\le \mu_1$,  $(\mathbf{H_1})$ and $(\mathbf{H_2})$ hold. Assume that $u_0,~u_1\in H_0^1(\Omega)$ and $f_0\in L^2(\Omega\times (0,1))$, then there exists a unique solution $(u,z)$ of problem \eqref{2.1} satisfying 
\[u,u_t\in C([0,T);H_0^1(\Omega)),\quad z\in C([0,T);L^2(\Omega\times(0,1))), \]
for $T>0.$
\end{theorem}

Define the energy functional of problem \eqref{2.1} as follows
\begin{equation}
\label{2.2}
\begin{split}
E(t)&=\frac{1}{\rho+2}\|u_t\|_{\rho+2}^{\rho+2}+\frac12\Big(1-\int_0^tg(s)ds\Big)\|\nabla u\|_2^2+\frac12(g\circ\nabla u)(t)\\
&\quad+\frac12\|\nabla u_t\|_2^2+\frac{\xi}{2}\int_\Omega\int_0^1z^2(x,\kappa,t)d\kappa dx-\frac bp\|u\|_p^p,
\end{split}
\end{equation}
here $\xi$ is a positive constant so that 
\begin{equation}
\label{2.3}
\tau|\mu_2|\le\xi\le\tau(2\mu_1-|\mu_2|)
\end{equation}
and
\[(g\circ u)(t)=\int_0^tg(t-s)\|u(s)-u(t)\|_2^2ds.\]

\begin{lemma}[Lemma 3.1 in \cite{W2019}]\label{lem2.1}
$E(t)$ is a non-increasing function and 
\begin{equation}
\label{01}
\begin{split}
E'(t)&\le -\omega(\|u_t\|_2^2+\|z(x,1,t)\|_2^2)+\frac 12(g'\circ \nabla u)(t)-\frac 12 g(t)\|\nabla u\|_2^2\\
&\le -\omega(\|u_t\|_2^2+\|z(x,1,t)\|_2^2)\le 0\quad\text{ for all }t\ge 0,
\end{split}
\end{equation}
where $\omega=\min\Big\{\mu_1-\frac{\xi}{2\tau}-\frac{|\mu_2|}{2},\frac{\xi}{2\tau}-\frac{|\mu_2|}{2}\Big\}\ge 0.$
\end{lemma}

\begin{lemma}\label{lem1}
If $u$ is a solution for problem $(\ref{2.1})$ and
\[E(0)<E_1=\frac{p-2}{2p}\sigma_1^2,\quad l\|\nabla u_0\|_2^2<\sigma_1^2,\]
here $\sigma_1=b^{-\frac{1}{p-2}}B_1^{-\frac{p}{p-2}}$, $B_1=\frac{c_s^p}{l^{\frac p2}}$, then there exists a positive constant $\sigma_2$ satisfying $0<\sigma_2<\sigma_1$ such that
\begin{equation}\label{2}
l\|\nabla u\|_2^2+(g\circ\nabla u)(t)\le\sigma_2^2\quad\text{ for all } t\ge0.
\end{equation}
\end{lemma}
\begin{proof}
Taking the combination of \eqref{2.2} and \eqref{02} with the embedding $H_0^1(\Omega)\hookrightarrow L^p(\Omega)$, one has
\begin{equation}
\label{3}
\begin{split}
E(t)&\ge \frac l2\|\nabla u\|_2^2+\frac12(g\circ\nabla u)(t)-\frac {bB_1^p}{p}\Big(l^{\frac12}\|\nabla u\|_2\Big)^p\\
&\ge F\Big(\sqrt{l\|\nabla u\|_2^2+(g\circ\nabla u)(t)}\Big),
\end{split}
\end{equation}
where  $F(x)=\frac 12x^2-\frac{bB_1^p}{p}x^p\text{ for } x>0.$ 
We know that $F$ is strictly increasing in $(0,\sigma_1)$, strictly decreasing in $(\sigma_1,\infty)$, and $F$ has a maximum at $\sigma_1$ with the maximum value $E_1$. 
Since $E(0)<E_1$, there exists a $\sigma_2<\sigma_1$  such that $F(\sigma_2)=E(0)$. Set $\sigma_0:=\sqrt{l\|\nabla u_0\|_2^2}$, recall \eqref{3},  then $F(\sigma_0)\leq E(0)=F(\sigma_2)$, which implies $\sigma_0\le \sigma_2$ since the given condition $\sigma_0^2<\sigma_1^2$. In order to complete the proof of \eqref{2}, we suppose by contradiction that for some $t^0>0$,
\[\sigma(t^0)=\sqrt{l\|\nabla u(t_0)\|_2^2+(g\circ\nabla u)(t_0)}>\sigma_2.\] The continuity of $\sqrt{l\|\nabla u\|_2^2+(g\circ\nabla u)(t)}$ illustrates that we may choose $t^0$
such that $\sigma_1>\sigma(t^0)>\sigma_2$, then we have
\[E(0)=F(\sigma_2)<F(\sigma(t^0))\leq E(t^0).\]
This is a contradiction since Lemma \ref{lem2.1}. 
\end{proof}

\begin{lemma}\label{lem2}
Under all the conditions of Lemma $\ref{lem1}$,  there exists a positive constant $\mathcal{D}$ such that  for all $t\ge 0$,
\begin{equation}
\label{4}
\|u\|_p^p\le \mathcal{D}E(t)\le \mathcal{D}E(0),
\end{equation}
\begin{equation}
\label{5}
\begin{split}
&\frac{1}{\rho+2}\|u_t\|_{\rho+2}^{\rho+2}+\frac12\Big(1-\int_0^tg(s)ds\Big)\|\nabla u\|_2^2+\frac12(g\circ\nabla u)(t)+\frac12\|\nabla u_t\|_2^2\\
&\quad+\frac{\xi}{2}\int_\Omega\int_0^1z^2(x,\kappa,t)d\kappa dx\le \mathcal{D}E(t)\le \mathcal{D}E(0).
\end{split}
\end{equation}

\end{lemma}
\begin{proof}
Using the embedding $H_0^1(\Omega)\hookrightarrow L^p(\Omega)$, \eqref{2.2} and \eqref{2}, we have
\begin{equation*}
\begin{split}
\frac bp\|u\|_p^p&\le \frac {bB_1^p}{p}\Big(l^{\frac 12}\|\nabla u\|_2\Big)^p\le \frac {bB_1^p}{p}\Big(l\|\nabla u\|_2^2+(g\circ\nabla u)(t)\Big)^{\frac {p-2}{2}}[l\|\nabla u\|_2^2+(g\circ\nabla u)(t)]\\
&\le \frac {2bB_1^p}{p}\Big(l\|\nabla u\|_2^2+(g\circ\nabla u)(t)\Big)^{\frac {p-2}{2}} \Big(E(t)+\frac bp\|u\|_p^p\Big)\\
&\le \frac {2bB_1^p}{p}\sigma_2^{p-2} \Big(E(t)+\frac bp\|u\|_p^p\Big),
\end{split}
\end{equation*}
which yields \eqref{4} with
\begin{equation}
\label{03}
\mathcal{D}=\frac{2pB_1^p\sigma_2^{p-2}}{p-2bB_1^p\sigma_2^{p-2}}>0.
\end{equation}
One has \eqref{5} by combining \eqref{4} with \eqref{2.2}, 
\end{proof}

\begin{lemma}[Lemma 4.1 in \cite{BAM2020}]\label{lem2.2}
For $u\in H_0^1(\Omega)$, we have for all $t\ge 0$,
\begin{equation}
\label{2.4}
\int_\Omega\Big(\int_0^tg(t-s)(\nabla u(s)-\nabla u(t))ds\Big)^2dx\le C_\alpha(h_\alpha\circ \nabla u)(t)
\end{equation}
where, for any $0<\alpha<1$,
\begin{equation}
\label{2.5}
C_\alpha=\int_0^\infty\frac{g^2(s)}{\alpha g(s)-g'(s)}ds\text{ and }h_\alpha(t)=\alpha g(t)-g'(t).
\end{equation}
\end{lemma}
Let us follow from the proof of Lemma 4.1 in \cite{BAM2020}, we have in fact 
\begin{equation}
\label{04}
\int_\Omega\Big(\int_0^tg(t-s)(u(s)-u(t))ds\Big)^2dx\le C_\alpha(h_\alpha\circ u)(t).
\end{equation}

\begin{lemma}[Lemma 2.2 in \cite{CB2021}]\label{lem01}
There exist positive constants $\gamma$ and $t_1$ such that 
\begin{equation}
\label{05}
g'(t)\le-\gamma g(t)\quad\text{for }t\in[0,t_1].
\end{equation}
\end{lemma}

\begin{lemma}\label{lem2.3}
Let $u$ be a solution of problem \eqref{2.1}, then the functional 
\begin{equation}
\label{2.6}
I_1(t)=\frac {1}{\rho+1}\int_\Omega |u_t|^\rho u_tudx+\int_\Omega \nabla u_t\nabla udx,
\end{equation}
satisfies,  for $\varepsilon>0$ and for all $t\ge 0$,
\begin{equation}
\label{2.7}
\begin{split}
I_1'(t)&\le\frac{1}{\rho+1}\|u_t\|_{\rho+2}^{\rho+2}-\Big[l-\Big(1+\frac{\mu_1}{\lambda_1}+\frac{\mu_2}{\lambda_1}\Big)\varepsilon\Big]\|\nabla u\|_2^2+\frac{1}{4\varepsilon}C_\alpha(h_\alpha\circ \nabla u)(t)\\
&\quad+\frac{|\mu_2|}{4\varepsilon}\|z(x,1,t)\|_2^2+b\|u\|_p^p+\|\nabla u_t\|_2^2+\frac{\mu_1}{4\varepsilon}\|u_t\|_2^2.
\end{split}
\end{equation}
\end{lemma}
\begin{proof}
Multiplying the first identity in problem \eqref{2.1} by $u$, integrating on $x$ over $\Omega$, and then using integration by parts, we give
\begin{equation}
\label{2.8}
\begin{split}
&\int_\Omega|u_t|^\rho u_{tt}udx+\int_\Omega \nabla u_{tt}\nabla udx\\
&\quad=-\Big(1-\int_0^tg(s)ds\Big)\|\nabla u\|_2^2+\int_\Omega \nabla u(t)\int_0^tg(t-s)\nabla (u(s)-u(t))ds\\
&\quad\quad-\int_\Omega \mu_1u_t(x,t)udx-\int_\Omega \mu_2 z(x,1,t)udx+b\|u\|_p^p.
\end{split}
\end{equation}
Differentiating \eqref{2.6} on $t$, and using \eqref{2.8}, one has
\begin{equation}
\label{2.9}
\begin{split}
I_1'(t)&=\frac{1}{\rho+1}\|u_t\|_{\rho+2}^{\rho+2}+\int_\Omega|u_t|^\rho u_{tt}udx+\int_\Omega \nabla u_{tt}\nabla udx+\|\nabla u_t\|_2^2\\
&=\frac{1}{\rho+1}\|u_t\|_{\rho+2}^{\rho+2}-\Big(1-\int_0^tg(s)ds\Big)\|\nabla u\|_2^2\\
&\quad+\int_\Omega \nabla u(t)\int_0^tg(t-s)\nabla (u(s)- u(t))ds\\
&\quad-\int_\Omega \mu_1u_t(x,t)udx-\int_\Omega \mu_2 z(x,1,t)udx+b\|u\|_p^p+\|\nabla u_t\|_2^2.
\end{split}
\end{equation}
Applying Cauchy's inequality with $\varepsilon>0$ and $\lambda_1\|u\|_2^2\le \|\nabla u\|_2^2$, it follows that 
\begin{equation}
\begin{split}
\label{2.10}
-\int_\Omega \mu_1u_t(x,t)udx&\le \frac{\mu_1}{4\varepsilon}\|u_t\|_2^2+\mu_1\varepsilon\|u\|_2^2\le \frac{\mu_1}{4\varepsilon}\|u_t\|_2^2+\frac{\mu_1\varepsilon}{\lambda_1}\|\nabla u\|_2^2,
\end{split}
\end{equation}
\begin{equation}
\label{2.101}
-\int_\Omega \mu_2 z(x,1,t)udx\le \frac{|\mu_2|}{4\varepsilon}\|z(x,1,t)\|_2^2+\frac{|\mu_2|\varepsilon}{\lambda_1}\|\nabla u\|_2^2.
\end{equation}
It follows from Cauchy's inequality with $\varepsilon>0$ and \eqref{2.4} that 
\begin{equation}
\label{2.11}
\int_\Omega \nabla u(t)\int_0^tg(t-s)(\nabla u(s)-\nabla u(t))ds\le \varepsilon\|\nabla u\|_2^2+\frac{1}{4\varepsilon}C_\alpha(h\circ \nabla u)(t).
\end{equation}
Inserting \eqref{2.10}-\eqref{2.11} into \eqref{2.9}, we obtain \eqref{2.7}.
\end{proof}

\begin{lemma}\label{lem2.5}
Under all the conditions of Lemma $\ref{lem1}$, let $u$ be a solution of problem \eqref{2.1}, then the functional 
\begin{equation}
\label{2.13}
I_2(t)=\int_\Omega\Big(\Delta u_t-\frac{1}{\rho+1}|u_t|^\rho u_t\Big)\int_0^t g(t-s)(u(t)-u(s))dsdx,
\end{equation}
satisfies, for $\delta>0$ and for all $t\ge 0$,
\begin{equation}
\label{2.14}
\begin{split}
I_2'(t)&\le B_1\|\nabla u\|_2^2+B_2(h_\alpha\circ \nabla u)(t)+\Big[B_3-\int_0^tg(s)ds\Big] \|\nabla u_t\|_2^2\\
&\quad+\delta\|z(x,1,t)\|_2^2-\int_0^tg(s)ds\cdot\frac{1}{\rho+1}\|u_t\|_{\rho+2}^{\rho+2},
\end{split}
\end{equation}
here $B_1$, $B_2$ and $B_3$ are positive constants depending on $\delta$ shown in \eqref{b}.
\end{lemma}
\begin{proof}
Differentiating \eqref{2.13} on $t$, and using  the first identity in problem \eqref{2.1} and integration by parts, we have
\begin{equation}
\label{2.15}
\begin{split}
I_2'(t)&=\int_\Omega\Big(\Delta u_{tt}-|u_t|^\rho u_{tt}\Big)\int_0^t g(t-s)(u(t)-u(s))dsdx\\
&\quad+\int_\Omega\Big(\Delta u_t-\frac{1}{\rho+1}|u_t|^\rho u_t\Big)\int_0^t g_t(t-s)(u(t)-u(s))dsdx\\
&\quad+\int_\Omega\Big(\Delta u_t-\frac{1}{\rho+1}|u_t|^\rho u_t\Big)\int_0^t g(t-s)u_t(t)dsdx\\
&=\int_\Omega \nabla u(t)\int_0^tg(t-s)\nabla (u(t)-u(s))ds\\
&\quad-\int_\Omega \int_0^tg(t-s)\nabla u(s)ds\int_0^tg(t-s)\nabla (u(t)-u(s))dsdx\\
&\quad+\int_\Omega \mu_1u_t(x,t)\int_0^t g(t-s)(u(t)-u(s))dsdx\\
&\quad+\int_\Omega \mu_2 z(x,1,t)\int_0^t g(t-s)(u(t)-u(s))dsdx\\
&\quad-b\int_\Omega |u|^{p-2}u\int_0^t g(t-s)(u(t)-u(s))dsdx\\
&\quad -\int_\Omega\nabla  u_t\int_0^t g_t(t-s)\nabla(u(t)-u(s))dsdx\\
&\quad-\int_\Omega\frac{1}{\rho+1}|u_t|^\rho u_t\int_0^t g_t(t-s)(u(t)-u(s))dsdx\\
&\quad-\int_0^tg(s)ds\cdot\|\nabla u_t\|_2^2-\int_0^tg(s)ds\cdot\frac{1}{\rho+1}\|u_t\|_{\rho+2}^{\rho+2}\\
&= J_1+J_2+\cdots+J_8+J_9.
\end{split}
\end{equation}
It is direct from Cauchy's inequality with $\delta>0$  and \eqref{2.4} that
\begin{equation}
\label{add2.16}
J_1\le \delta\|\nabla u\|_2^2+\frac{1}{4\delta}C_\alpha(h_\alpha\circ \nabla u)(t),
\end{equation}
\begin{equation}
\label{2.16}
\begin{split}
J_2&\le \delta\int_\Omega\Big(\int_0^t g(t-s)|\nabla u(s)-\nabla u(t)|+|\nabla u(t)|ds\Big)^2dx\\
&\quad +\frac{1}{4\delta}\int_\Omega\Big(\int_0^t g(t-s)|\nabla u(s)-\nabla u(t)|ds\Big)^2dx\\
&\le 2\delta\int_\Omega\Big(\int_0^t g(t-s)|\nabla u(s)-\nabla u(t)|ds\Big)^2dx+2\delta(1-l)^2\|\nabla u(t)\|_2^2\\
&\quad +\frac{1}{4\delta}\int_\Omega\Big(\int_0^t g(t-s)|\nabla u(s)-\nabla u(t)|ds\Big)^2dx\\
&\le \Big(2\delta+\frac{1}{4\delta}\Big)C_\alpha(h_\alpha\circ \nabla u)(t)+2\delta(1-l)^2\|\nabla u\|_2^2.
\end{split}
\end{equation}
Cauchy's inequality with $\delta>0$ and \eqref{04} yield 
\begin{equation}
\label{6}
\begin{split}
J_3\le \delta \|u_t\|_2^2+\frac{\mu_1^2}{4\delta}C_\alpha(h_\alpha\circ u)(t)\le\frac{\delta}{\lambda_1} \|\nabla u_t\|_2^2+\frac{\mu_1^2}{4\delta\lambda_1}C_\alpha(h_\alpha\circ \nabla u)(t),
\end{split}
\end{equation}
\begin{equation}
\label{7}
\begin{split}
J_4&\le \delta\|z(x,1,t)\|_2^2+\frac{\mu_2^2}{4\delta\lambda_1}C_\alpha(h_\alpha\circ \nabla u)(t).
\end{split}
\end{equation}
It follows from Cauchy's inequality with $\delta>0$,  \eqref{04}, the embedding $H_0^1(\Omega)\hookrightarrow L^{2(p-1)}(\Omega)$ and \eqref{5} that 
\begin{equation}
\label{1}
\begin{split}
J_5&\le    b\delta\|u\|_{2(p-1)}^{2(p-1)}+\frac{b}{4\delta}\int_\Omega \Big(\int_0^t g(t-s)(u(t)-u(s))ds\Big)^2dx\\
& \le b\delta\|u\|_{2(p-1)}^{2(p-1)}+\frac{b}{4\delta}C_\alpha (h_\alpha\circ u)(t)\\
&\le b\delta c_s^{2(p-1)}\Big(\frac{2\mathcal{D}}{l}E(0)\Big)^{p-2}\|\nabla u\|_2^2+\frac{b}{4\delta\lambda_1}C_\alpha (h_\alpha\circ \nabla u)(t).
\end{split}
\end{equation}
Recalling the definition of  $g'(t)$ in \eqref{2.5}, and using Cauchy's inequality with $\delta>0$, \eqref{2.4} and H\"older's inequality, one obtains 
\begin{equation}
\label{2.17}
\begin{split}
J_6&=-\int_\Omega\nabla  u_t\int_0^t \alpha g(t-s)\nabla(u(t)-u(s))dsdx\\
&\quad+\int_\Omega\nabla  u_t\int_0^t h_\alpha(t-s)\nabla(u(t)-u(s))dsdx\\
&\le \delta\|\nabla u_t\|_2^2+\frac{\alpha^2}{4\delta}\int_\Omega\Big(\int_0^t g(t-s)\nabla(u(t)-u(s))ds\Big)^2dx\\
&\quad+\delta\|\nabla u_t\|_2^2+\frac{1}{4\delta}\int_\Omega\Big(\int_0^t  h_\alpha(t-s)\nabla(u(t)-u(s))ds\Big)^2dx\\
&\le 2\delta\|\nabla u_t\|_2^2+\frac{\alpha^2}{4\delta}C_\alpha(h_\alpha\circ \nabla u)(t)\\
&\quad+\frac{1}{4\delta}\int_0^t  h_\alpha(s)ds\int_0^t  h_\alpha(t-s)\|\nabla(u(t)-u(s))\|_2^2ds\\
&\le 2\delta\|\nabla u_t\|_2^2+\Big(\frac{\alpha^2}{4\delta}C_\alpha+\frac{\alpha(1-l)+g(0)}{4\delta}\Big)(h_\alpha\circ \nabla u)(t).
\end{split}
\end{equation}
Similarly, we get
\begin{equation}
\label{2.19}
\begin{split}
J_7&=-\int_\Omega\frac{1}{\rho+1}|u_t|^\rho u_t\int_0^t \alpha g(t-s)(u(t)-u(s))dsdx\\
&\quad+\int_\Omega\frac{1}{\rho+1}|u_t|^\rho u_t\int_0^t h_\alpha(t-s)(u(t)-u(s))dsdx\\
&\le \frac{\delta}{\rho+1}\|u_t\|_{2(\rho+1)}^{2(\rho+1)}+\frac{\alpha^2}{4(\rho+1)\delta}\int_\Omega\Big(\int_0^t g(t-s)(u(t)-u(s))ds\Big)^2dx\\
&\quad+\frac{\delta}{\rho+1}\|u_t\|_{2(\rho+1)}^{2(\rho+1)}+\frac{1}{4(\rho+1)\delta}\int_\Omega\Big(\int_0^t  h_\alpha(t-s)(u(t)-u(s))ds\Big)^2dx\\
&\le \frac{2\delta}{\rho+1}\|u_t\|_{2(\rho+1)}^{2(\rho+1)}+\frac{\alpha^2}{4(\rho+1)\delta}C_\alpha(h_\alpha\circ  u)(t)\\
&\quad+\frac{1}{4(\rho+1)\delta}\int_0^t  h_\alpha(s)ds\int_0^t  h_\alpha(t-s)\|u(t)-u(s)\|_2^2ds\\
&\le \frac{2\delta}{\rho+1}c_s^{2(\rho+1)}\Big(2\mathcal{D}E(0)\Big)^\frac{\rho}{2}\|\nabla u_t\|_2^2\\
&\quad+\Big(\frac{\alpha^2}{4(\rho+1)\delta}C_\alpha+\frac{\alpha(1-l)+g(0)}{4(\rho+1)\delta}\Big)\frac{1}{\lambda_1}(h_\alpha\circ  \nabla u)(t).
\end{split}
\end{equation}

Inserting \eqref{2.16}-\eqref{2.19} into \eqref{2.15}, one has 
\begin{equation*}
\begin{split}
I_2'(t)&\le B_1\|\nabla u\|_2^2+B_2(h_\alpha\circ \nabla u)(t)+\Big[B_3-\int_0^tg(s)ds\Big] \|\nabla u_t\|_2^2\\
&\quad+\delta\|z(x,1,t)\|_2^2-\int_0^tg(s)ds\cdot\frac{1}{\rho+1}\|u_t\|_{\rho+2}^{\rho+2},
\end{split}
\end{equation*}
with 
\begin{equation}
\label{b}
\begin{cases}
B_1=\delta +2\delta(1-l)^2+b\delta c_s^{2(p-1)}\Big(\frac{2\mathcal{D}}{l}E(0)\Big)^{p-2};\\
B_2=\Big[\frac{1}{2\delta}+2\delta+\frac{\mu_1^2}{4\delta\lambda_1}+\frac{\mu_2^2}{4\delta\lambda_1}+\frac{b}{4\delta\lambda_1}+\frac{\alpha^2}{4\delta}+\frac{\alpha^2}{4(\rho+1)\delta\lambda_1}\Big]C_\alpha\\
\quad\quad\quad+\frac{\alpha(1-l)+g(0)}{4\delta}+\frac{\alpha(1-l)+g(0)}{4(\rho+1)\delta\lambda_1};\\
B_3=\frac{\delta}{\lambda_1}+2\delta+\frac{2\delta}{\rho+1}c_s^{2(\rho+1)}\Big(2\mathcal{D}E(0)\Big)^\frac{\rho}{2}.
\end{cases}
\end{equation}

\end{proof}

\begin{lemma}\label{lem3}
The functional 
\begin{equation}
\label{8}
I_3(t)=\int_0^1e^{-2\tau\kappa}\|z(x,\kappa,t)\|_2^2d\kappa
\end{equation}
satisfies  for all $t\ge 0$,
\begin{equation}
\label{9}
I_3'(t)=-2I_3(t)+\frac 1\tau\|u_t\|_2^2-\frac{e^{-2\tau}}{\tau}\|z(x,1,t)\|_2^2.
\end{equation}
\end{lemma}
\begin{proof}
Differentiating \eqref{8} on $t$, and using the second identity in \eqref{2.1}, it is direct that 
\begin{equation}
\label{18}
\begin{split}
I_3'(t)&=\int_0^12e^{-2\tau\kappa}\int_\Omega z(x,\kappa,t)z_t(x,\kappa,t)dxd\kappa\\
&=-\int_0^1\frac 2\tau e^{-2\tau\kappa}\int_\Omega z(x,\kappa,t)z_k(x,\kappa,t)dxd\kappa\\
&=-\frac 1\tau\int_\Omega\int_0^1 \Big(\frac{d}{d\kappa}e^{-2\tau\kappa}z^2(x,\kappa,t)+2\tau e^{-2\tau\kappa}z^2(x,\kappa,t)\Big) d\kappa dx\\
&=-2I_3(t)+\frac 1\tau\|u_t\|_2^2-\frac{e^{-2\tau}}{\tau}\|z(x,1,t)\|_2^2.
\end{split}
\end{equation}

\end{proof}

\begin{lemma}[Lemma 3.4 in \cite{M2018}]\label{lem4}
The functional 
\begin{equation}
\label{10}
I_4(t)=\int_0^t f(t-s)\|\nabla u(s)\|_2^2ds
\end{equation}
satisfies for all $t\ge 0$,
\begin{equation}
\label{11}
I_4'(t)\le 3(1-l)\|\nabla u\|_2^2-\frac 12 (g\circ \nabla u)(t),
\end{equation}
where $f(t)=\int_t^\infty g(s)ds.$
\end{lemma}

\section{Stability results}
In this section, we will present and prove  the decay results of the energy functional $E(t)$ based on the lemmas in Section 2. To begin with, we 
define a functional 
\begin{equation}
\label{12}
L(t)=ME(t)+\sum_{i=1}^{3}N_iI_i(t),
\end{equation}
where $M,~N_1,~N_2,~N_3$ are positive constants. The following lemma is shown to illustrate that $L(t)$ is equivalent to $E(t)$.

\begin{lemma}\label{lem5}
 Under all the conditions of Lemma $\ref{lem1}$, assume that $M$ is enough large, then there exist two positive constants $\beta_1$ and $\beta_2$ such that
\begin{equation}
\label{13}
\beta_1E(t) \le L(t) \le \beta_2E(t).
\end{equation}
\end{lemma}
\begin{proof}
Recalling the definition of $I_1(t)$ in \eqref{2.6}, using Young's inequality and Cauchy's inequality, and then applying the embedding $H_0^1(\Omega)\hookrightarrow L^{\rho+2}(\Omega)$ and \eqref{5}, it is not hard to give 
\begin{equation}
\label{14}
\begin{split}
|I_1(t)|&\le\frac {1}{\rho+2}\|u_t\|^{\rho+2}_{\rho+2}+ \frac {1}{(\rho+1)(\rho+2)}\|u\|^{\rho+2}_{\rho+2}+\frac 12 \|\nabla u_t\|_2^2+\frac 12 \|\nabla u\|_2^2\\
&\le \frac {1}{\rho+2}\|u_t\|^{\rho+2}_{\rho+2}+ \Big[\frac {c_s^{\rho+2}}{(\rho+1)(\rho+2)}\Big(\frac{2\mathcal{D}}{l}E(0)\Big)^{\frac \rho2}+\frac12\Big] \|\nabla u\|_2^2+\frac 12 \|\nabla u_t\|_2^2.
\end{split}
\end{equation}
Recalling the definition of $I_2(t)$ in \eqref{2.13}, using integration by parts, Cauchy's inequality, Young's inequality and H\"older's inequality,  we give 
\begin{equation}
\label{15}
\begin{split}
|I_2(t)|&\le\frac 12\|\nabla u_t\|_2^2+\frac 12\int_\Omega\Big(\int_0^tg(t-s)(\nabla u(t)-\nabla u(s))ds\Big)^2dx\\
&\quad+\frac {1}{\rho+2}\|u_t\|^{\rho+2}_{\rho+2}+\frac{1}{(\rho+1)(\rho+2)}\int_\Omega\Big(\int_0^tg(t-s)( u(t)-u(s))ds\Big)^{\rho+2}dx\\
&\le \frac 12\|\nabla u_t\|_2^2+ \frac{1-l}{2}(g\circ \nabla u)(t)+\frac {1}{\rho+2}\|u_t\|^{\rho+2}_{\rho+2}\\
&\quad+\frac{1}{(\rho+1)(\rho+2)}(1-l)^{\rho+1}c_s^{\rho+2}\Big(\frac{2\mathcal{D}}{l}E(0)\Big)^{\frac \rho2}(g\circ \nabla u)(t),
\end{split}
\end{equation}
here we have used
\begin{equation}
\label{16}
\begin{split}
&\int_\Omega\Big(\int_0^tg(t-s)( u(t)-u(s))ds\Big)^{\rho+2}dx\\
&\quad\le\int_\Omega\Big(\int_0^t(g(t-s))^{\frac{\rho+1}{\rho+2}}(g(t-s))^{\frac{1}{\rho+2}}( u(t)-u(s))ds\Big)^{\rho+2}dx\\
&\quad\le\Big(\int_0^tg(s)ds\Big)^{\rho+1}\int_0^tg(t-s)\|u(t)-u(s))\|_{\rho+2}^{\rho+2}ds\\
&\quad\le (1-l)^{\rho+1}c_s^{\rho+2}\int_0^tg(t-s)\|\nabla u(t)-\nabla u(s))\|_2^{\rho+2}ds\\
&\quad\le (1-l)^{\rho+1}c_s^{\rho+2}\Big(\frac{2\mathcal{D}}{l}E(0)\Big)^{\frac \rho2}(g\circ \nabla u)(t).
\end{split}
\end{equation}
Therefore, it follows from \eqref{5} that  
\begin{equation*}
\begin{split}
|L(t)-ME(t)|=|\sum_{i=1}^{3}N_iI_i(t)|\le CE(t),
\end{split}
\end{equation*}
here $C$ is some positive constant.
\end{proof}

\begin{lemma}\label{lem6}
Under all the conditions of Lemma $\ref{lem1}$, for $|\mu_2|<\mu_1$, the functional $L(t)$ defined in \eqref{12} satisfies, for $t\ge t_1$
\begin{equation}
\label{17}
\begin{split}
L'(t)&\le -C_1\|u_t\|_{\rho+2}^{\rho+2}-C_2\|\nabla u_t\|_2^2-4(1-l)\|\nabla u\|_2^2\\
&\quad+N_1b\|u\|_p^p+\frac14(g\circ \nabla u)(t)-2N_3\int_0^1\|z(x,\kappa,t)\|_2^2d\kappa.
\end{split}
\end{equation}
where $C_1,~C_2$ are positive constants given in \eqref{c}.
\end{lemma}

\begin{proof}
Taking the combination of \eqref{01}, \eqref{2.7} and \eqref{2.14} with \eqref{9},  recalling \eqref{2.5},  and applying $g_1=\int_0^{t_1}g(s)ds\le \int_0^{t}g(s)ds$ for $t\ge t_1$, one has

\begin{equation}
\label{19}
\begin{split}
L'(t)
&\le -M\omega(\|u_t\|_2^2+\|z(x,1,t)\|_2^2)+\frac M2(g'\circ \nabla u)(t)-\frac M2 g(t)\|\nabla u\|_2^2\\
&\quad+\frac{N_1}{\rho+1}\|u_t\|_{\rho+2}^{\rho+2}-N_1\Big[l-\Big(1+\frac{\mu_1}{\lambda_1}+\frac{\mu_2}{\lambda_1}\Big)\varepsilon\Big]\|\nabla u\|_2^2+\frac{N_1}{4\varepsilon}C_\alpha(h_\alpha\circ \nabla u)(t)\\
&\quad+\frac{N_1|\mu_2|}{4\varepsilon}\|z(x,1,t)\|_2^2+N_1b\|u\|_p^p+N_1\|\nabla u_t\|_2^2+N_1\frac{\mu_1}{4\varepsilon}\|u_t\|_2^2\\
&\quad+N_2B_1\|\nabla u\|_2^2+N_2B_2(h_\alpha\circ \nabla u)(t)+N_2\Big[B_3-\int_0^tg(s)ds\Big] \|\nabla u_t\|_2^2\\
&\quad+N_2\delta\|z(x,1,t)\|_2^2-N_2\int_0^tg(s)ds\cdot\frac{1}{\rho+1}\|u_t\|_{\rho+2}^{\rho+2}\\
&\quad-2N_3I_3(t)+\frac {N_3}{\tau}\|u_t\|_2^2-\frac{N_3e^{-2\tau}}{\tau}\|z(x,1,t)\|_2^2\\
&\le -C_1\|u_t\|_{\rho+2}^{\rho+2}-C_2\|\nabla u_t\|_2^2-C_3\|\nabla u\|_2^2-C_4\|z(x,1,t)\|_2^2-C_5(h_\alpha\circ \nabla u)(t)\\
&\quad+N_1b\|u\|_p^p+\frac {\alpha M}{2}(g\circ \nabla u)(t)-C_6\|u_t\|_2^2-2N_3\int_0^1\|z(x,\kappa,t)\|_2^2d\kappa
\end{split}
\end{equation}
with 
\begin{equation}
\label{c}
\begin{cases}
C_1=N_2g_1\cdot\frac{1}{\rho+1}-\frac{N_1}{\rho+1};\\
C_2=N_2\Big[g_1-\Big(\frac{\delta}{\lambda_1}+2\delta+\frac{2\delta}{\rho+1}c_s^{2(\rho+1)}\Big(2\mathcal{D}E(0)\Big)^\frac{\rho}{2}\Big)\Big]-N_1;\\
C_3=N_1\Big[l-\Big(1+\frac{\mu_1}{\lambda_1}+\frac{|\mu_2|}{\lambda_1}\Big)\varepsilon\Big]-N_2\Big[\delta +2\delta (1-l)^2+b\delta c_s^{2(p-1)}\Big(\frac{2\mathcal{D}}{l}E(0)\Big)^{p-2}\Big];\\
C_4=\omega M+N_3\frac{e^{-2\tau}}{\tau}-N_1\frac{|\mu_2|}{4\varepsilon}-N_2\delta;\\
C_5=\frac M2-N_1\frac{1}{4\varepsilon}C_\alpha-N_2\Big[\Big(\frac{1}{2\delta}+2\delta+\frac{\mu_1^2}{4\delta\lambda_1}+\frac{\mu_2^2}{4\delta\lambda_1}+\frac{b}{4\delta\lambda_1}+\frac{\alpha^2}{4\delta}+\frac{\alpha^2}{4(\rho+1)\delta\lambda_1}\Big)C_\alpha;\\
\quad\quad\quad+\frac{\alpha(1-l)+g(0)}{4\delta}+\frac{\alpha(1-l)+g(0)}{4(\rho+1)\delta\lambda_1}\Big];\\
C_6=\omega M-\frac{N_3}{\tau}-N_1\frac{\mu_1}{4\varepsilon},
\end{cases}
\end{equation}
where we have used the values of $B_1$, $B_2$ and $B_3$ defined in \eqref{b}.

Next, we choose $\delta$ such that 
\begin{equation*}
\begin{split}
\delta&<\Bigg\{\frac{lg_1}{16\Big[1 +2 (1-l)^2+b c_s^{2(p-1)}\Big(\frac{2\mathcal{D}}{l}E(0)\Big)^{p-2}\Big]},\frac{lg_1}{1024(1-l)^2},\\
&\quad\quad\frac{5g_1}{8\Big(\frac{1}{\lambda_1}+2+\frac{2}{\rho+1}c_s^{2(\rho+1)}\Big(2\mathcal{D}E(0)\Big)^\frac{\rho}{2}\Big)}\Bigg\}.
\end{split}
\end{equation*}
Let us choose $N_1=\frac 38g_1N_2$, then 
\[C_1=N_2g_1\cdot\frac{1}{\rho+1}-\frac 38g_1N_2\frac{1}{\rho+1}=\frac 58g_1N_2\frac{1}{\rho+1}>0,\quad C_2>0.\]
Let us fix 
\[\varepsilon=\frac{3l}{4}\frac{1}{1+\frac{\mu_1}{\lambda_1}+\frac{|\mu_2|}{\lambda_1}},\]
then
\[C_3=\frac{N_1l}{4}-N_2\Big[\delta +2\delta (1-l)^2+b\delta c_s^{2(p-1)}\Big(\frac{2\mathcal{D}}{l}E(0)\Big)^{p-2}\Big]>\frac{l}{32}g_1N_2>0.\]
By taking $N_2=\frac{1}{8\delta(1-l)}$, we get
\[C_3>\frac{l}{32}g_1N_2=\frac{lg_1}{256\delta(1-l)}>4(1-l).\]

Since $g'(s)\le 0$, one has $\frac{\alpha g^2(s)}{\alpha g(s)-g'(s)}\le g(s)$, further we get 
\[\lim_{\alpha\to 0^+}\alpha C_\alpha=\lim_{\alpha\to 0^+}\int_0^\infty \frac{\alpha g^2(s)}{\alpha g(s)-g'(s)} ds=0.\]
Thus, there exists $0<\alpha_0<1$ so that if $\alpha<\alpha_0$, then 
\[\alpha C_\alpha<\frac{1}{8\Big[N_2\Big(\frac{1}{2\delta}+2\delta+\frac{\mu_1^2}{4\delta\lambda_1}+\frac{\mu_2^2}{4\delta\lambda_1}+\frac{b}{4\delta\lambda_1}+\frac{\alpha^2}{4\delta}+\frac{\alpha^2}{4(\rho+1)\delta\lambda_1}\Big)+N_1\frac{1}{4\varepsilon}\Big]}.\]
Let us choose $M$ sufficiently large such that for $\alpha=\frac{1}{2M}$,
\[C_5=\frac M4-N_2\Big[\frac{\alpha(1-l)+g(0)}{4\delta}+\frac{\alpha(1-l)+g(0)}{4(\rho+1)\delta\lambda_1}\Big]>0,\]
\[C_4=\omega M+N_3\frac{e^{-2\tau}}{\tau}-N_1\frac{|\mu_2|}{4\varepsilon}-N_2\delta>0,\]
\[C_6=\omega M-\frac{N_3}{\tau}-N_1\frac{\mu_1}{4\varepsilon}>0,\]
here we have used $\omega>0$ given in Lemma \ref{lem2.1} since $|\mu_2|<\mu_1$.  

Based on the above discussion, one has from \eqref{19}
\begin{equation*}
\begin{split}
L'(t)&\le -C_1\|u_t\|_{\rho+2}^{\rho+2}-C_2\|\nabla u_t\|_2^2-4(1-l)\|\nabla u\|_2^2\\
&\quad+N_1b\|u\|_p^p+\frac 14(g\circ \nabla u)(t)-2N_3\int_0^1\|z(x,\kappa,t)\|_2^2d\kappa.
\end{split}
\end{equation*}

\end{proof}

Now, we give the following stability results.
\begin{theorem}\label{thm3.3}
For $|\mu_2|<\mu_1$, suppose that $(\mathbf{H_1}),~(\mathbf{H_2})$ hold and \[E(0)<E_1,\quad l\|\nabla u_0\|_2^2<\sigma_1^2,\] then there exist positive constants $k_1, k_2, k_3$ and $k_4$ such
that the solution of problem \eqref{1111} satisfies for all $t\ge t_1$, 
\begin{equation*}
E(t)\le 
\begin{cases}k_1e^{-k_2\int_{t_1}^t\zeta (s)ds}&\quad \text{for }G\text{ is linear};\\
 k_4G_1^{-1}\Big(k_3\int_{t_1}^t\zeta (s)ds\Big)&\quad \text{for }G\text{ is nonlinear},
\end{cases}
\end{equation*}
where $E_1$ and $\sigma_1$ are shown in Lemma $\ref{lem1}$, $G_1(t)=\int_t^r\frac{1}{sG'(s)}ds$ is strictly decreasing and convex in $(0,r]$ with $\lim_{t\to 0}G_1(t)=+\infty.$
\end{theorem}
\begin{proof}
Using \eqref{05} and \eqref{01}, one has, for $t\ge t_1$,
\begin{equation}
\label{06}
\begin{split}
&\int_0^{t_1}g(s)\|\nabla u(t)-\nabla u(t-s)\|_2^2ds\\
&\quad\le-\frac{1}{\gamma}\int_0^{t_1}g'(s)\|\nabla u(t)-\nabla u(t-s)\|_2^2ds\le -cE'(t)
\end{split}
\end{equation}
here  $c$ is used to denote a generic positive constant throughout this proof.  Define a functional $F(t)$ that is obviously equivalent to $E(t)$ as follows 
\[F(t)=L(t)+cE(t),\]
then based on \eqref{17}, \eqref{2.2}, \eqref{06}, for  some $m>0$ and for any $t\ge t_1$, we have 
\begin{equation}
\label{07}
\begin{split}
F'(t)&\le -C_1\|u_t\|_{\rho+2}^{\rho+2}-C_2\|\nabla u_t\|_2^2-4(1-l)\|\nabla u\|_2^2\\
&\quad+N_1b\|u\|_p^p+\frac 14(g\circ \nabla u)(t)-2N_3\int_0^1\|z(x,\kappa,t)\|_2^2d\kappa+cE'(t)\\
&\le -mE(t)-\Big(\frac{bc}{p}-N_1b\Big)\|u\|_p^p+c(g\circ \nabla u)(t)+cE'(t)\\
&\le -mE(t)+c\int_{t_1}^{t}g(s)\|\nabla u(t)-\nabla u(t-s)\|_2^2ds,
\end{split}
\end{equation}
here we have chosen $N_1$ so small  that $\frac{bc}{p}-N_1b>0$.

In what follows, we will discuss in two cases. 

\textsc{Case 1: $G$ is linear} Multiplying \eqref{07} by $\zeta(t)$, using $(\mathbf{H_1})$ and \eqref{01}, one gives
\begin{equation*}
\begin{split}
\zeta(t) F'(t)&\le -m\zeta(t) E(t)+c\zeta(t)\int_{t_1}^{t} g(s)\|\nabla u(t)-\nabla u(t-s)\|_2^2ds\\
&\le -m\zeta(t) E(t)-c\int_{t_1}^{t}g'(s)\|\nabla u(t)-\nabla u(t-s)\|_2^2ds\\
&\le -m\zeta(t) E(t)-cE'(t),
\end{split}
\end{equation*}
which implies
\[(\zeta(t) F(t)+cE(t))'\le -m\zeta(t) E(t)\quad \text{for }t\ge t_1.\]
Integrating the above inequality over $(t_1,t)$, and using the fact that $\zeta (t)F(t)+cE(t)$ is equivalent to $E(t)$, one has
\[E(t)\le k_1e^{-k_2\int_{t_1}^t\zeta (s)ds}\quad \text{for }t\ge t_1,\]
here $k_1$ and $k_2$ are constants.

\textsc{Case 2: $G$ is nonlinear} Define a functional
\[H(t)=L(t)+I_4(t).\]
Taking the combination of Lemma \ref{lem5} and the non-negativity of $E(t)$ obtained by Lemma \ref{lem2} with the definition of $I_4(t)$ in \eqref{10}
, it is not difficult to get the non-negativity of $H(t)$. It follows from \eqref{17} and \eqref{11} that for some $m_1>0$ and $t\ge t_1$,
\begin{equation*}
\begin{split}
H'(t)&=L'(t)+I'_4(t)\\
&\le-C_1\|u_t\|_{\rho+2}^{\rho+2}-C_2\|\nabla u_t\|_2^2-(1-l)\|\nabla u\|_2^2\\
&\quad+N_1b\|u\|_p^p-\frac 14(g\circ \nabla u)(t)-2N_3\int_0^1\|z(x,\kappa,t)\|_2^2d\kappa\le -m_1E(t).
\end{split}
\end{equation*}
Integrating the above inequality over $(t_1,t)$ yields
\[m_1\int_{t_1}^tE(s)ds\le H(t_1)-H(t)\le H(t_1),\]
which implies 
\begin{equation}
\label{08}
\int_0^\infty E(s)ds<+\infty.
\end{equation}

Define 
\[\lambda (t)=p\int_{t_1}^t\|\nabla u(t)-\nabla u(t-s)\|_2^2ds,\]
by using \eqref{5}, then we give
\begin{equation*}
\begin{split}
\lambda (t)&\le 2p\int_0^t\Big(\|\nabla u(t)\|_2^2+\|\nabla u(t-s)\|_2^2\Big)ds\\
&\le \frac{8p\mathcal{D}}{l}\int_0^t\Big(E(t)+E(t-s)\Big)ds\le \frac{16p\mathcal{D}}{l}\int_0^tE(t-s)ds\\
&=\frac{16p\mathcal{D}}{l}\int_0^tE(s)ds\le\int_0^\infty E(s)ds<+\infty.
\end{split}
\end{equation*}
Thus, we can choose $p$ so small  that  for $t\ge t_1$,
\begin{equation}
\label{08}
\lambda(t)<1.
\end{equation}
It is direct that 
\begin{equation}
\label{09}
G(\theta z) \le \theta G(z)\quad \text{for }0 \le \theta \le1\text{ and } z\in (0, r],
\end{equation}
since $G$ is strictly convex on $(0, r]$ and $G(0) = 0$. Based on \eqref{add02}, \eqref{08}, \eqref{09} and Jensen's inequality, one gives
\begin{equation*}
\begin{split}
I(t)&= \frac{1}{p\lambda (t)}\int_{t_1}^t\lambda (t)(-g'(s))p\|\nabla u(t)-\nabla u(t-s)\|_2^2ds\\
&\ge \frac{1}{p\lambda (t)}\int_{t_1}^t\lambda (t)\zeta(s)G(g(s))p\|\nabla u(t)-\nabla u(t-s)\|_2^2ds\\
&\ge \frac{\zeta(t)}{p\lambda (t)}\int_{t_1}^t\overline{G}(\lambda (t)g(s))p\|\nabla u(t)-\nabla u(t-s)\|_2^2ds\\
&\ge \frac{\zeta(t)}{p}\overline{G}\Big(p\int_{t_1}^tg(s)\|\nabla u(t)-\nabla u(t-s)\|_2^2ds\Big),
\end{split}
\end{equation*}
which yields
\begin{equation}
\label{010}
\int_{t_1}^tg(s)\|\nabla u(t)-\nabla u(t-s)\|_2^2ds\le\frac{1}{p}\overline{G}^{-1}\Big(\frac{pI(t)}{\zeta(t)}\Big),
\end{equation}
where $G$ has an extension $\overline{G}$ which is a strictly increasing and strictly convex $C^2$ function on $(0, +\infty)$  as in Remark 2.1 \cite{CB2021}. 
Therefore, \eqref{07} becomes 
\begin{equation}
\label{011}
F'(t)\le-mE(t)+\frac{c}{p}\overline{G}^{-1}\Big(\frac{pI(t)}{\zeta(t)}\Big).
\end{equation}
Let us define the functional 
\[F_1(t)=\overline{G}'\Big(\frac{r_1E(t)}{E(0)}\Big)F(t)+E(t)\]
with $0<r_1<r,$ then $F_1$ is equivalent to $E$ and 
\begin{equation}
\label{012}
\begin{split}
F_1'(t)&=\frac{r_1E'(t)}{E(0)}\overline{G}''\Big(\frac{r_1E(t)}{E(0)}\Big)F(t)+\overline{G}'\Big(\frac{r_1E(t)}{E(0)}\Big)F'(t)+E'(t)\\
&\le -mE(t)\overline{G}'\Big(\frac{r_1E(t)}{E(0)}\Big)+\frac{c}{p}\overline{G}^{-1}\Big(\frac{pI(t)}{\zeta(t)}\Big)\overline{G}'\Big(\frac{r_1E(t)}{E(0)}\Big)+E'(t)
\end{split}
\end{equation}
by using \eqref{011}, \eqref{01}, $G'>0$ and $G''>0$. Let $\overline{G}^*$ be the convex conjugate of $G$ in the sense of Young in \cite{A1989}, which is given by
\begin{equation}
\label{013}
\overline{G}^*(s) = s(\overline{G}')^{-1}(s)-\overline{G}
\Big[(\overline{G}')^{-1}(s)\Big]
\end{equation}
and it satisfies the following Young's inequality
\begin{equation}
\label{014}
AB\le\overline{G}^*(A) + \overline{G}(B).
\end{equation}
Choosing 
\[A=\overline{G}'\Big(\frac{r_1E(t)}{E(0)}\Big)\text{ and }B=\overline{G}^{-1}\Big(\frac{pI(t)}{\zeta(t)}\Big),\]
then  using \eqref{014}, \eqref{013}  and the non-negativity of $\overline{G}, $ \eqref{012} becomes
\begin{equation}
\label{015}
\begin{split}
F_1'(t)&\le -mE(t)\overline{G}'\Big(\frac{r_1E(t)}{E(0)}\Big)+\frac{c}{p}\overline{G}^{-1}\Big(\frac{pI(t)}{\zeta(t)}\Big)\overline{G}'\Big(\frac{r_1E(t)}{E(0)}\Big)+E'(t)\\
&\le -mE(t)\overline{G}'\Big(\frac{r_1E(t)}{E(0)}\Big)+\frac{c}{p}\overline{G}^*\Big(\overline{G}'\Big(\frac{r_1E(t)}{E(0)}\Big)\Big)+c\frac{I(t)}{\zeta(t)}+E'(t)\\
&\le -mE(t)\overline{G}'\Big(\frac{r_1E(t)}{E(0)}\Big)+\frac{c}{p}\frac{r_1E(t)}{E(0)}\overline{G}'\Big(\frac{r_1E(t)}{E(0)}\Big)\\
&\quad-\overline{G}\Big(\frac{r_1E(t)}{E(0)}\Big)+c\frac{I(t)}{\zeta(t)}+E'(t)\\
&\le -mE(t)\overline{G}'\Big(\frac{r_1E(t)}{E(0)}\Big)+\frac{c}{p}\frac{r_1E(t)}{E(0)}\overline{G}'\Big(\frac{r_1E(t)}{E(0)}\Big)+c\frac{I(t)}{\zeta(t)}+E'(t).
\end{split}
\end{equation}
Note that  $\eqref{01}$ implies
\begin{equation*}
I(t)\le\int_{t_1}^t-g'(s)\|\nabla u(t)-\nabla u(t-s)\|_2^2ds \le -2E'(t),
\end{equation*}
then one has\begin{equation}
\zeta (t)F_1'(t)\le -m\zeta (t)E(t)G'\Big(\frac{r_1E(t)}{E(0)}\Big)+\frac{c}{p}\zeta(t)\frac{r_1E(t)}{E(0)}G'\Big(\frac{r_1E(t)}{E(0)}\Big)-cE'(t)
\end{equation}
 by multiplying \eqref{015} by $\zeta (t)$ and by using the fact \[\overline{G}'\Big(\frac{r_1E(t)}{E(0)}\Big)=G'\Big(\frac{r_1E(t)}{E(0)}\Big).\]

Define the functional $F_2(t)=\zeta (t)F_1(t)+cE(t)$ which  is equivalent to $E(t)$, which means 
\begin{equation}
\label{016}
\gamma_1 F_2(t)\le E(t)\le \gamma_2 F_2(t)
\end{equation}
for some $\gamma_1$ and $\gamma_2$. Under a suitable choice of $r_1$ and for a positive constant $k$, we have
\begin{equation}
\label{017}
F_2'(t)\le -k\zeta(t)\frac{E(t)}{E(0)}G'\Big(\frac{r_1E(t)}{E(0)}\Big)=-k\zeta(t)G_2\Big(\frac{E(t)}{E(0)}\Big)
\end{equation}
with $G_2(t)=tG'(r_1t)$. Obviously, $G_2$ and $G_2'$ are positive in $(0,1]$ since \[G_2'(t)=G'(r_1t)+r_1t^2G''(r_1t)\] and the convexity of $G$ in $(0,r].$ 
\eqref{017} and \eqref{016} imply
\begin{equation}
\label{018}
\Big(\frac{\gamma_1F_2(t)}{E(0)}\Big)'\le -k\zeta(t)\frac{\gamma_1}{E(0)}G_2\Big(\frac{E(t)}{E(0)}\Big)\le -k_3\zeta(t)G_2\Big(\frac{\gamma_1F_2(t)}{E(0)}\Big)
\end{equation}
with $k_3=k\frac{\gamma_1}{E(0)}$. Setting $R(t)=\frac{\gamma_1F_2(t)}{E(0)}$, and then integrating \eqref{018} over $(t_1,t)$, one has 
\[\int_{t_1}^t-\frac{R'(s)}{G_2(R(s))}ds\ge \int_{t_1}^tk_3\zeta(s)ds.\]
Since $r_1R(t_1)<r$, we have
\[G_1(r_1R(t))=\int_{r_1R(t)}^{r_1R(t_1)}\frac{1}{sG'(s)}ds\ge k_3\int_{t_1}^t\zeta (s)ds.\]
It is noted that $G_1$ is strictly decreasing function on $(0, r]$ and $\lim_{t\to 0}G_1(t)=+\infty$ in Theorem \ref{thm3.3}, then
\[R(t) \le \frac{1}{r_1}G_1^{-1}\Big(k_3\int_{t_1}^t\zeta (s)ds\Big).\]
Since  $R(t)$ is equivalent to $E(t)$, further one obtains
\[E(t) \le k_4G_1^{-1}\Big(k_3\int_{t_1}^t\zeta (s)ds\Big)\]
with $k_4=\frac{1}{r_1}.$

This completes the proof of this theorem.
\end{proof}

\section{Possible generalizations}
Our results can be generalized to the following initial boundary value problem with the time-varying delay $\tau(t)$
\begin{equation*}
\begin{cases}
      |u_t|^\rho u_{tt}-\Delta u-\Delta u_{tt}+\int_0^tg(t-s)\Delta u(s)ds\\
      \quad+\mu_1u_t(x,t)+\mu_2 u_t(x,t-\tau(t))=b|u|^{p-2}u& (x,t)\in\Omega \times(0,\infty),  \\
      u_t(x,t-\tau(0))=f_0(x,t-\tau(0))&(x,t)\in\Omega\times(0,\tau(0)),  \\
      u(x,0)=u_0(x),~u_t(x,0)=u_1(x)&x\in\Omega,\\
      u(x,t)=0&(x,t)\in\partial\Omega\times[0,\infty),
\end{cases}
\end{equation*}
equipped with the following assumptions in addition to $(\mathbf{H_1})$ and $(\mathbf{H_2})$: 

$(A_1)$ the function $\tau\in W^{2,\infty}([0,T])$, for all $T>0$ such that 
\[0<\tau_0\le \tau(t)\le\tau_1\quad\text{for all }t>0,\]
\[\tau'(t)\le d<1\quad\text{for all }t>0,\]
where $\tau_0$ and $\tau_1$ are positive numbers.

$(A_2)$ the coefficients of delay and dissipation satisfy
\[|\mu_2|\le \frac{2(1-d)}{2-d}\mu_1.\]
The above assumptions in fact  are given in \cite{CB2021}. To be more precise, we need to define the new energy functional 
\begin{equation*}
\begin{split}
E(t)&=\frac{1}{\rho+2}\|u_t\|_{\rho+2}^{\rho+2}+\frac12\Big(1-\int_0^tg(s)ds\Big)\|\nabla u\|_2^2+\frac12(g\circ\nabla u)(t)\\
&\quad+\frac12\|\nabla u_t\|_2^2+\frac{\xi}{2}\tau(t)\int_\Omega\int_0^1z^2(x,\kappa,t)d\kappa dx-\frac bp\|u\|_p^p,
\end{split}
\end{equation*}
where $\xi$ satisfies 
\[\frac{|\mu_2|}{1-d}\le \xi\le 2\mu_1-|\mu_2|.\]
We also need to replace $I_3(t)$ in Lemma \ref{lem3} by 
\[I_3(t)=\int_0^1e^{-2\tau(t)\kappa}\|z(x,\kappa,t)\|_2^2d\kappa,\]
further,  one has
\[I_3'(t)\le -2I_3(t)-\frac{(1-d)e^{-2\tau_1}}{\tau_1}\|z(x,1,t)\|_2^2+\frac{1}{\tau_0}\|u_t\|_2^2\]
as shown in Lemma 2.8 of \cite{CB2021}. Based on the above changes, we believe that it is possible to get the similar results in Theorem \ref{thm3.3} by following the steps in this paper.

In addition, we have a conjecture that our results can also be extended to the  following initial boundary value problem:
\begin{equation*}
\label{1.1}
\begin{cases}
      |u_t|^\rho u_{tt}-\Delta u-\Delta u_{tt}+\int_0^tg(t-s)\Delta u(s)ds\\
      \quad+\mu_1|u_t(x,t)|^{m-2}u_t(x,t)\\
      \quad+\mu_2 |u_t(x,t-\tau)|^{m-2}u_t(x,t-\tau)=b|u|^{p-2}u& (x,t)\in\Omega \times(0,\infty),  \\
      u_t(x,t-\tau)=f_0(x,t-\tau)&(x,t)\in\Omega\times(0,\tau),  \\
      u(x,0)=u_0(x),~u_t(x,0)=u_1(x)&x\in\Omega,\\
      u(x,t)=0&(x,t)\in\partial\Omega\times[0,\infty).
\end{cases}
\end{equation*}

\subsection*{Acknowledgements}
The author would like  to express her sincere gratitude to Professor Wenjie Gao and  Professor Bin Guo in Jilin University  for their support and constant encouragement.

\end{document}